\documentclass[12pt,regno]{amsart}
\usepackage{epsf,epsfig,amsmath}
\usepackage{amssymb,latexsym}

\usepackage{amsthm} 
\usepackage{amsfonts}
\usepackage{graphicx,psfrag}

\numberwithin{equation}{section}

\newtheorem{thm}{Theorem}[section]
\newtheorem{pro}[thm]{Proposition}
\newtheorem{lem}[thm]{Lemma}

\newtheorem{rem}[thm]{Remark}

\newtheorem*{cor*}{Corollary}
\newtheorem*{thm*}{Theorem}

\title[Bruhat intervals and parabolic cosets ]{
Bruhat intervals and parabolic cosets in arbitrary Coxeter groups}

\author{Mario Marietti}

\address{Dipartimento  di Ingegneria Industriale e Scienze Matematiche, Universit\`a Politecnica delle Marche, Via Brecce Bianche, 60131 Ancona,  Italy}

\email{m.marietti@univpm.it}


\begin{document}

\begin{abstract}
In  [Journal of Pure and Applied Algebra {224} (2020), no 12,  106449], V. Mazorchuk and R. Mr{\dj}en (with some help by A. Hultman) prove that, given a Weyl group, the intersection of a Bruhat interval with a parabolic coset has a unique maximal element and a unique minimal element. We show that such intersections are actually Bruhat intervals also in the case of an arbitrary  Coxeter group.
\end{abstract}

\maketitle

\section{Introduction}
Let $W$ be a Weyl group and $S$ be a set of Coxeter generators of $W$. Consider the Bruhat order with respect to $S$. Let $J$ be a subset of $S$ and $x,y,u \in W$.   In \cite{MazMrd}, V. Mazorchuk and R. Mr{\dj}en prove that, if the intersection of the Bruhat interval $[x,y]$ with the parabolic coset $uW_j$ is nonempty, then it has a unique maximal element  (see \cite[Lemma 3]{MazMrd}, for which the authors acknowledge help by A. Hultman) and a unique minimal element (see \cite[Lemma 5]{MazMrd}). The proof for the existence of a unique maximal element works also in the case of an arbitrary  Coxeter group. On the other hand, the proof for the existence of a unique minimal element makes use of the longest element of $W$, which does not exist in infinite Coxeter groups. In this short note, we give an alternative proof that does not assume the finiteness and works for all Coxeter groups.

\bigskip 

\section{Notation and preliminaries}

This section reviews the background material that is needed  in the proof of Theorem~\ref{teorema}.

Let $(W,S)$ be an arbitrary Coxeter system. The group $W$, under Bruhat order (see, e.g.,  \cite[\S 2.1]{BB} or \cite[\S 5.9]{Hum}), sometimes also called Bruhat-Chevalley order,  is a graded partially ordered set  having the length function $\ell$ as its rank function. This means that $W$ has a minimum, which is the identity element $e$, and the   function $\ell$  satisfies $\ell (e)=0$ and $\ell (y) =\ell (x)+1$ for all $x,y \in W$ with $x \lhd y$. Here $x \lhd y$, as well as  $y \rhd x$, means that  the Bruhat interval $[x,y]$ coincides with $\{x,y\}$. 

Given $w\in W$, we let $D_R(w)$ denote the right descent set $\{ s \in S : \; \ell(w  s) < \ell(w ) \}$ of $w$. Given a subset $J$ of $S$, we let $W_J$ denote  the parabolic subgroup of $W$ generated by $J$ and $W^J$ denote the set $\{ w \in W \, : \; D_{R}(w)\subseteq S\setminus J \}$ of minimal left coset representatives. For $x\in W$, we let $W_{\leq x} =\{ w \in W \, : \; w \leq x \}$ and $W_{\geq x} =\{ w \in W \, : \; w \geq x \}$.

The following results are well known (see, e.g.,  \cite[Proposition~2.2.7]{BB} or \cite[Proposition~5.9]{Hum} for the first one,  \cite[\S 2.4]{BB} or \cite[\S 1.10]{Hum} for the second one, and \cite[\S 3.2]{BB} for the lattice property of the weak Bruhat order implying the third one). 
\begin{lem}[Lifting Property]
\label{ll}
Let $s\in S$ and $u,w\in W$, $u\leq w$. Then
\begin{enumerate}
\item[(i)]
\label{i}
 if $s\in D_R(w)$ and $s\in D_R(u)$ then $us\leq ws$,
\item[(ii)] 
\label{ii}
if $s\notin D_R(w)$ and $s \notin D_R(u)$ then $us\leq ws$,
\item[(iii)] 
\label{iii}
if $s\in D_R(w)$ and $s\notin D_R(u)$ then $us\leq w$ and $u\leq ws$.
\end{enumerate}
\end{lem}
\begin{pro}
\label{fattorizzo}
Let $J \subseteq S$. 
Every $w \in W$ has a unique factorization $w=w^{J} \cdot w_{J}$ 
with $w^{J} \in W^{J}$ and $w_{J} \in W_{J}$; for this factorization, $\ell(w)=\ell(w^{J})+\ell(w_{J})$.
\end{pro}

\begin{pro}
\label{discese}
Let $x\in W$. If $s_1,s_2 \in D_R(x)$, then the order of the product $s_1  s_2$ is finite. 
\end{pro}
Symmetrically, left versions of Lemma~\ref{ll} and Proposition~\ref{fattorizzo} hold, as well as of the following well-known (and immediate to prove) result:
\begin{eqnarray}
\label{minoreinparabolico}
 v\leq w  \implies  v^J\leq w^J 
\end{eqnarray}

\section{Arbitrary Coxeter groups}

\begin{thm}
\label{teorema}
Let $(W,S)$ be an arbitrary Coxeter system. The intersection of a Bruhat interval with a parabolic coset is a Bruhat interval.  
\end{thm}
\begin{proof}
By (\ref{minoreinparabolico})
, it is sufficient to prove that the intersection of a Bruhat interval with a parabolic coset, if nonempty, has a unique maximal element and a unique minimal element. The proof  in \cite[Lemma 3]{MazMrd} for the existence of a unique maximal element  in Weyl groups works also for arbitrary Coxeter groups.

 Let us prove the existence of a unique minimal element in the case of a left coset (the mirrored argument works for a right coset). It is sufficient to prove the following claim:  given $x,u \in W$ and a subset $J$ of $S$, the intersection $W_{\geq x} \cap uW_J$, if nonempty,  has a unique minimal element. 
 
 Let $x$, $u$, and $J$ be as in the claim and suppose $W_{\geq x} \cap uW_J\neq \emptyset$. We may also suppose $u\in W^J$. We use induction on $\ell(x)$. If  $\ell(x)=0$, then $x=e$ and $W_{\geq e} \cap uW_J$ has a unique minimal element, which is $u$. 
 
 Suppose $\ell(x)>0$ and, towards a contradiction, suppose that $m_1$ and $m_2$ are two distinct minimal elements of $W_{\geq x} \cap uW_J$.

Fix $i\in \{1,2\}$. Clearly $m_i\neq u$. Hence, there exists $s_i\in D_R(m_i)\cap J$. The minimality of $m_i$ implies $xs_i \lhd x$ since otherwise, by Lemma~\ref{ll}(iii), we would have $m_is_i \in W_{\geq x} \cap uW_J$. Let $m^i= \min  (W_{\geq xs_i} \cap uW_J)$, which exists by the induction hypothesis and satisfies $m^i\leq m_1$ and $m^i\leq m_2$ since $m_1$ and $m_2$ both belong to $W_{\geq xs_i} \cap uW_J$. Furthermore. $m^is_i \rhd m^i$ since otherwise, by  Lemma~\ref{ll}(iii), we would have $m^i \in W_{\geq x} \cap uW_J$, against the minimality of $m_1$ and $m_2$. Again Lemma~\ref{ll}(iii) implies $m^is_i \leq m_i$ while Lemma~\ref{ll}(ii) implies  $m^is_i  \in W_{\geq x} \cap uW_J$; by the minimality of $m_i$, we have $m^is_i=m_i$. Furthermore, if we let $\bar{i}$ be the element of the singleton $\{1,2\} \setminus \{i\}$, then we have $m_{\bar{i}}s_i \rhd m_{\bar{i}}$ since otherwise the same argument  would imply that also $m_{\bar{i}}$ coincides with $m^is_i$, but $m_1\neq m_2$.

The four relations $m_1s_1\lhd m_1$, $m_1s_1\leq m_2$, $m_2s_2\lhd m_2$, $m_2s_2\leq m_1$ imply $\ell(m_1)=\ell(m_2)$, $m_1s_1\lhd m_2$, and $m_2s_2\lhd m_1$. By a repeated use of Lemma~\ref{ll}(i), we conclude that there exists $w\in W^{\{s_1,s_2\}}$ such that $m_1$ and $m_2$ belong  to the coset $w W_{\{s_1,s_2\}}$. 

Notice that $xs_1 \lhd x$ and $xs_2 \lhd x$; hence $W_{\{s_1,s_2\}}$ is finite by Lemma~\ref{discese}, and $x$ is the top element of the coset $xW_{\{s_1,s_2\}}$.

Lemma~\ref{ll}(i) implies

 $$w \,  (\underbrace{\cdots s_i s_{\bar{i}} s_i}_{\text{$h$ terms}}) \geq x \iff 
 w \,  (\underbrace{\cdots s_i s_{\bar{i}} }_{\text{$h-1$ terms}}) \geq x s_i \iff 
 \cdots \iff
w\geq x \, (\underbrace{s_i s_{\bar{i}} s_i \cdots }_{\text{$h$ terms}})  $$
  for each $h\in \mathbb N$ smaller than, or equal to, the rank of $W_{\{s_1,s_2\}}$.
  
 Hence, by  the four relations $m_1 \geq x$, $m_2 \geq x$, $m_1s_1 \not \geq x$, and  $m_2s_2 \not \geq x$, we conclude that the intersection $W_{\leq w} \cap xW_{\{s_1,s_2\}}$ has two distinct maximal elements, which is a contradiction since we know that $W_{\leq w} \cap xW_{\{s_1,s_2\}}$ has a unique maximal element.
\end{proof}

\begin{rem}
\begin{enumerate}
\item Prior to \cite{MazMrd}, the special case of the existence of a unique maximal element in $W_{\geq x} \cap W_J$ , for all $x\in W$, is proved in \cite[Lemma 7]{Hom74}.
\item The proof of Theorem~\ref{teorema} provides another evidence of the fundamental role of (parabolic) dihedral subgroups and dihedral intervals in understanding the combinatorial properties of Coxeter groups (see, for example, \cite{BCM1}, \cite{CM1}, \cite{CM2}, 
\cite{Dye2}, \cite{Dye3}, \cite{Dyepreprint}, \cite{Mtrans}, \cite{M}
).
\item  Fix a subset $J$ of $S$. Let $x\in W$ and $u_1,u_2\in W^J$. While $\min ( W_{\geq x} \cap u_1W_J) \leq \min ( W_{\geq x} \cap u_2W_J)$, as well as $\max ( W_{\geq x} \cap u_1W_J) \leq \max ( W_{\geq x} \cap u_2W_J)$, clearly implies $u_1 \leq u_2$ by  (\ref{minoreinparabolico}), the converse does not hold in general, as one may see already   in type $A_2$. 
\end{enumerate}
\end{rem}

\end{document}